\documentclass[a4paper,11pt,oneside,bookmarks]{article}
\usepackage[utf8x]{inputenc}
\usepackage{graphicx}
\usepackage{wrapfig}
\usepackage[pdftex,bookmarks]{hyperref}
\hypersetup{
pdftitle={Simons inequality},
pdfauthor={José M. Zapata}
}

\usepackage{fancyhdr}
\usepackage[english]{babel}
\usepackage{amsmath}
\usepackage{amsthm}
\usepackage{amsfonts}

\usepackage{amscd}

\usepackage{epigraph}

\usepackage{amssymb}

\newcommand{\N}{\mathbb{N}}
\newcommand{\R}{\mathbb{R}}

\newcommand{\PP}{\mathbb{P}}

\DeclareMathOperator*{\esssup}{ess.sup\,}
\DeclareMathOperator*{\essinf}{ess.inf\,}
\DeclareMathOperator*{\esslimsup}{ess.limsup\, }
\DeclareMathOperator*{\essliminf}{ess.liminf\, }

\newtheorem{defn}{Definition}[section]

\newtheorem{prop}{Proposition}[section]

\newtheorem{thm}{Theorem}[section]

\providecommand{\keywords}[1]{\textbf{\textit{Keywords:}} #1}

\begin{document}


\title{A random version of Simons' inequality}

\author{José M. Zapata}

\date{\today}
\maketitle


\begin{abstract}
The purpose of this paper is to provide a random version of Simons' inequality.
\end{abstract}

 

\section*{Introduction}

Based on the original proof of Simons’ inequality given in \cite{key-1}, we establish a random version of Simons’ inequality, so that instead of using functions taking their values in a real interval, we deal with functions taking their values in a ball of random radius.

In the first part of this paper we establish some necessary notions, and later, we enunciate the theorem and provide the corresponding proof.
   
\section{Preliminary notions}

Given a probability space $\left(\Omega,\mathcal{F},\PP\right)$, which will be fixed for the rest of this paper, we consider the set $L^{0} \left(\Omega,\mathcal{F},\PP\right)$, the set of equivalence classes of  real valued $\mathcal{F}$-measurable random variables, which will be denoted simply as $L^{0}$.

It is known that the triple $\left(L^{0},+,\cdot\right)$ endowed with the partial order of the almost sure dominance is a lattice ordered ring.

We say ``$X\geq Y$`` if $\PP\left( X\geq Y \right)=1$.
Likewise, we say ``$X>Y$'', if $\PP\left( X> Y \right)=1$. 

And, given $A\in \mathcal{F}$, we say that $X>Y$ (respectively,  $X \geq Y$) on $A$, if $\PP\left(X>Y \mid A\right)=1$ (respectively , if $\PP\left(X \geq Y \mid  A \right)=1$).

%
%
%

We can also define the set $\bar{L^{0}}$, the set of equivalence classes of  $\mathcal{F}$-measurable random variables taking values in $\bar{\R}=\R\cup\{\pm\infty\}$, and extend the partial order of the almost sure dominance to $\bar{L^{0}}$. 

~\\

In A.5 of \cite{key-3} is proved the proposition below

\begin{prop}
Let $\phi$  be a subset of $L^{0}$, then
\begin{enumerate}
\item There exists $Y^{*}\in\bar{L^{0}}$ such that $Y^{*}\geq Y$ for all 
$Y\in\phi,$ and such that any other $Y'$ satisfying the same, verifies $Y'\geq Y^{*}.$
\item Suppose that  $\phi$ is directed upwards. Then there exists a increasing sequence  $Y_{1}\leq Y_{2}\leq...$ in $\phi,$ such that 
 $Y_{n}$ converges to $Y^{*}$ almost surely.
\end{enumerate}
\end{prop}

\begin{defn}
Under the conditions of the previous proposition, 
$Y^{*}$ is called essential supremum of $\phi$, and we write

\[
\esssup\phi=\underset{Y\in\phi}{\esssup Y}:=Y^{*}
\]

The essential infimum of $\phi$ is defined as

\[
\essinf\phi=\underset{Y\in\phi}{\essinf Y}:=\underset{Y\in\phi}{-\esssup\left(-Y\right)}
\]
\end{defn}
\begin{defn}
Given a sequence $\left\{ Y_{n}\right\}_{n\in\N}$ in
$L^{0}$, we define the essential  limit inferior (respectively, the essential  limit superior) as 
\[
\underset{n\rightarrow\infty}{\essliminf} Y_{n}:=\underset{n}{\esssup}\underset{m\geq n}{\essinf} Y_{m}
\]
(respectively, $\underset{n\rightarrow\infty}{\esslimsup} Y_{n}:=\underset{n}{\essinf}\underset{m\geq n}{\esssup} Y_{m}$).
\end{defn}

Likewise, let us give a couple of definitions.

\begin{defn}

Given $\varepsilon\in L_{++}^{0}$ we define
\[
B_{\varepsilon}:=\left\{ Y\in L^{0};\:\left|Y\right|\leq\varepsilon\right\} 
\]
the ball of radius $\varepsilon$ centered at $0\in L^{0}$.
\end{defn}

\begin{defn}
Given a subset $S$ of a $L^0$-module $E$, we define the $L^0$-convex hull of $S$ as the set below
\[
co_{L^0}(S):=\left\{\underset{i\in I}{\sum}{Y_i X_i} ; \: I \textnormal{ finite, }X_i \in E \textnormal{, } Y_i \in L^0_+\textnormal{, }\underset{i\in I}{\sum}{Y_i}=1 \right\}
\]
\end{defn}

\section{Random Simons' inequality}

Finally we are going to introduce the random version of Simons' inequality. The proof is an adaptation of the original proof of Simons' inequality from \cite{key-1}. 

The development is essentially the same as \cite{key-1}, but we must overcome the obstacle that the order of the almost sure dominance on $L^0$ is not a total order. This is obtained by defining piece wise random variables, or in other words, countable concatenations of $L^0$.   

\begin{thm}
\label{thm: Simons' inequality}
Let $E$ be a set, and let $\{f_n\}_{n\in\N}$ a sequence of functions on $E$ taking their values in $B_\varepsilon$ for some $\varepsilon \in L^0_{++}$. Let $S$ be a subset of $E$ such that for every sequence $\{Y_n\}_{n\in\N}$ in $L^0_{++}$ with $\sum_{n\geq 1} Y_n=1$, there exists $Z\in S$ such that 
\[
\underset{X\in E}{\esssup}\underset{n\geq 1}{\sum}{Y_n f_n (X)}=\underset{n\geq 1}{\sum}{Y_n f_n (Z)}.
\]
Then
\[
\underset{X\in S}{\esssup}{\underset{n\rightarrow\infty}{\esslimsup}f_n(X)}\geq \underset{f\in co_{L^0}(f_n)}{\essinf}\underset{X\in E}{\esssup}f(X)
\]
\end{thm}
\begin{proof}
Let 
\[
m:=\underset{f\in co_{L^0}(f_n)}{\essinf}\underset{X\in E}{\esssup}f(X)
\]
and for $Z\in S$
\[
u(Z):={\underset{n\rightarrow\infty}{\,\esslimsup}f_n(Z)}.
\]
We wish to show that
\begin{equation} 
\label{ineq}
\underset{Z\in S}{\esssup}u(Z)\geq m.
\end{equation}

Define $M:=\underset{n\geq 1}{\esssup} \underset{X\in E} \esssup f_n$.

Then, given $\delta \in L^0_{++}$ arbitrary, we can choose $\lambda \in L^0_{++}$ with $0<\lambda<1$ fulfilling
\begin{equation} 
\label{ineq2}
m-\delta(1+\lambda)-M\lambda\geq (m-2\delta)(1-\lambda).
\end{equation}  

Define $C_n:=co_{L^0}\left\{f_n; \: p \geq n \right\}$. Now, using induction, we are going to prove that, for each $n\in \N$, there exist $g_n \in C_n$ such that 

\begin{equation} 
\label{ineq3}
\gamma_n(g_n)\leq \underset{g\in C_n}{\essinf} \gamma_n(g) + \delta \left( \frac{\lambda}{2}\right)^n, 
\end{equation} 

with
\[
\gamma_n(h):=\underset{Z\in S}{\esssup}\left(\underset{p\leq n-1}{\sum}{\lambda^{p-1}g_p} + \lambda^{n-1}h \right)
\textnormal{, for } h\in C_n.
\] 

For $n>1$, let us considerate that we have got $g_1,g_2,...,g_{n-1}$ fulfilling \ref{ineq3} for $n-1$. We are going to find an appropriate $g_n$.  

We claim that the set
\[
\left\{\gamma_n(h) ;\: h\in C_n \right\}
\]  
is downwards directed.

Indeed, given $g$, $g'\in C_n$, define $A:=(\gamma_n(g)\leq \gamma_n(g'))$ and take $\hat{g}:={1_A}g+1_{A^c}g'$.
 
Then, $\hat{g}\in C_n$ and we have 
\[
1_A \gamma_n(\hat{g}) = 1_A \gamma_n(g) \textnormal{ and}
\] 
\[
1_{A^c} \gamma_n(\hat{g}) = 1_{A^c} \gamma_n(g'). 
\] 
Adding up both equations yields
\[
\gamma_n(\hat{g})=\gamma_n(g) \wedge \gamma_n(g').
\] 
 
Thereby, there exists a sequence $\left\{h_k\right\}_{k\in\N}$ in $C_n$ such that $\gamma_n(h_k)\searrow \underset{g\in C_n}{\essinf} \gamma_n(g)$.

Now, we consider the sequence of sets 
\[
A_{0}:=\phi, 
\]
\[
A_{k}:=\left(\gamma_n(h_k)<\underset{g\in C_n}{\essinf} \gamma_n(g) + \delta \left( \frac{\lambda}{2}\right)^n\right)-A_{k-1} \textnormal{ for } k>0.
\]

Thus, $\left\{A_k\right\}_{k\in\N}$ is a partition of $\Omega$ and we define 
\[
g_n:=\underset{k\geq 1}{\sum}1_{A_k} h_k. 
\]
Then, $g_n \in C_n$ and for each $k\in\N$,
\[
1_{A_k}\gamma_n(g_n)=1_{A_k}\gamma_n(h_k)\leq 
1_{A_k}\left[\underset{g\in C_n}{\essinf} \gamma_n(g) + \delta \left( \frac{\lambda}{2}\right)^n \right]
\]
and, therefore
\[
\gamma_n(g_n)\leq 
\underset{g\in C_n}{\essinf} \gamma_n(g) + \delta \left( \frac{\lambda}{2}\right)^n.
\] 
Now, since for each $n\in \N$
\[
\frac{g_n + \lambda g_{n+1}}{1+\lambda}\in C_n,
\]
 and applying \ref{ineq3}, it follows
\begin{equation} 
\label{ineq4}
\gamma_n(g_n)\leq \gamma_n\left(\frac{g_n + \lambda g_{n+1}}{1+\lambda}\right) + \delta \left( \frac{\lambda}{2}\right)^n. 
\end{equation} 
Let $s_0:=0$, $s_n:=\sum_{p\leq n}\lambda^{p-1}g_p$ for $n\geq 1$, and $s:=\sum_{n\geq 1}\lambda^{n-1}g_n$. Multiplying \ref{ineq4} by $(1+\lambda)$, 
\[
(1+\lambda)\underset{Z\in S}{\esssup} s_n \leq 
\underset{Z\in S}{\esssup}(\lambda s_{n-1} + s_{n+1}) +  \delta (1+\lambda) \left( \frac{\lambda}{2}\right)^n
\]   
\[
\leq \lambda \underset{Z\in S}{\esssup} s_{n-1} + \underset{Z\in S}{\esssup} s_{n+1} +  \delta (1+\lambda) \left( \frac{\lambda}{2}\right)^n,
\]
for all $n\geq 1$. 

Thus, 
\[
\lambda^{-n}(\underset{Z\in S}{\esssup} s_{n+1} - \underset{Z\in S}{\esssup} s_n) \geq \lambda^{-n+1} (\underset{Z\in S}{\esssup} s_{n} - \underset{Z\in S}{\esssup} s_{n-1}) - \frac{\delta (1+\lambda)}{2^n}.
\]
Since $\underset{Z\in S}{\esssup} s_{1} - \underset{Z\in S}{\esssup} s_0=\underset{Z\in S}{\esssup} s_{1}\geq m$, using the above inequality and induction it holds that, for all $n\geq 1$ 
\[
\lambda^{-n+1}(\underset{Z\in S}{\esssup} s_{n} - \underset{Z\in S}{\esssup} s_{n-1})\geq m - \delta(1+\lambda)\sum_{i=1}^{n-1}2^{-i}\geq
\]
\[
\geq m - \delta(1+\lambda). 
\]
Hence,
\[
\underset{Z\in S}{\esssup} s - \underset{Z\in S}{\esssup} s_{n-1} = 
\underset{p\geq n}{\sum}(\underset{Z\in S}{\esssup} s_{p} - \underset{Z\in S}{\esssup} s_{p-1}) \geq
\]
\[
\geq \underset{p\geq n}{\sum} \lambda^{p-1}[m - \delta(1+\lambda)],
\]
and it leads to
\begin{equation} 
\label{ineq5}
\underset{Z\in S}{\esssup} s - \underset{Z\in S}{\esssup} s_{n-1} \geq \frac{\lambda^{n-1}}{1-\lambda} [m - \delta(1+\lambda)].
\end{equation}

Now, applying the assumption to $(1-\lambda)s$, there exists $Z_0\in S$ such that 
\[
s(Z_0)= \underset{Z\in S}{\esssup} s.
\]
For each $n\in\N$, we have
\[
\lambda^{n-1}g_n=s(Z_0)-s_{n-1}(Z_0)-\underset{p\geq n+1}{\sum}\lambda^{p-1}g_p(Z_0)\geq
\]
\[
\geq \underset{Z\in S}{\esssup} s - \underset{Z\in S}{\esssup} s_{n-1} - \underset{p\geq n+1}{\sum}\lambda^{p-1}M\geq
\]
\[
\overset{\ref{ineq5}}{\geq} \frac{\lambda^{n-1}}{1-\lambda} [m - \delta(1+\lambda)] - \frac{\lambda^n}{1-\lambda}M.
\]

Thereby, by \ref{ineq2} it follows that $g_n(Z_0)\geq m - 2\delta$ for each $n\in\N$.

Finally, since $g_n\in C_n$,
\[
\underset{m\geq n}{\esssup} f_m(Z_0)=\underset{g \in C_n}{\esssup} g(Z_0)\geq
\]
\[
\geq g_n(Z_0)\geq m - 2\delta,
\]
and taking essential infimum on $n$
\[
\underset{n\rightarrow\infty}{\esslimsup}f_m(Z_0) =\underset{n}{\essinf}\underset{m\geq n}{\esssup} f_m(Z_0)\geq g_n(Z_0)\geq m - 2\delta.
\]
 The result follows since $\delta\in L^0_{++}$ is arbitrary.

\end{proof}


\begin{thebibliography}{1} 
\bibitem {key-2} Godefroy, Gilles. \emph{Some applications of Simons’ inequality.} Serdica Mathematical Journal 26.1 (2000): 59p-78p.
\bibitem{key-3} Schied, Alexander, and Hans Follmer. \emph{Stochastic Finance; An Introduction in Discrete Time.} Walter De Gruyter, 2011.
\bibitem {key-1} Simons, Stephen. \emph{A convergence theorem with boundary.} Pacific Journal of Mathematics 40.3 (1972): 703-708.

 
\end{thebibliography}
\end{document}